\documentclass{amsart}

\usepackage{amsmath}
\usepackage{amssymb}
\usepackage{amsthm}

\newtheorem{thm}{Theorem}[section]
\newtheorem{prop}[thm]{Proposition}
\newtheorem{lem}[thm]{Lemma}

\theoremstyle{definition}
\newtheorem{dfn}[thm]{Definition}
\newtheorem{problem}[thm]{Problem}

\theoremstyle{remark}
\newtheorem{rem}[thm]{Remark}
\newtheorem*{acknowledgement}{Acknowledgements}

\newcommand{\R}{\mathbb{R}}
\newcommand{\Z}{\mathbb{Z}}

\newcommand{\Mag}{\mathrm{Mag}}
\newcommand{\1}{1\!\!1}

\newcommand{\tr}{\mathrm{tr}}

\title[Monotonicity of magnitude function]
{On the monotonicity of magnitude functions of 
negative type finite semimetric spaces}

\author{Kiyonori Gomi}

\address{
Department of Mathematics, 
Institute of Science Tokyo,
2-12-1 Ookayama, Meguro-ku, Tokyo, 152-8551, Japan.}

\email{kgomi@math.titech.ac.jp}

\author{Hiroshi Tsuji}

\address{Department of Mathematics, 
Institute of Science Tokyo,
2-12-1 Ookayama, Meguro-ku, Tokyo, 152-8551, Japan.}

\email{tsujihiroshi@math.sci.isct.ac.jp}



\begin{document}

\begin{abstract}
Motivated by the conjectured monotonicity of the magnitude functions of metric spaces of negative type, we investigate the monotonicity problem for finite semimetric spaces, namely, ``metric spaces without the triangle inequality''. We completely resolve this problem by proving that the monotonicity holds for negative type semimetric spaces with at most three points, while constructing negative type semimetric spaces with at least four points whose magnitude functions are not monotonically increasing. We also prove monotonicity for all $4$-point metric spaces and construct $12$-point metric spaces of negative type whose magnitude functions are not monotonically increasing.
\end{abstract}

\maketitle

\tableofcontents


\section{Introduction}
\label{sec:introduction}

In an attempt to generalize the Euler characteristic of topological spaces and the cardinality of sets, Leinster introduced the notion of \textit{magnitude} to certain enriched categories \cite{L}. As a special case, the magnitude of a finite metric space is defined. The definition is simple: Let $d : X \times X \to [0, \infty)$ be a metric on a finite set $X$. The \textit{similarity matrix} or \textit{zeta matrix} of $X = (X, d)$ is defined by $Z_X = (e^{-d_{x, y}})_{x, y \in X}$. In the case where $Z_X$ is invertible, the magnitude $\Mag(X)$ is defined as the sum of the entries of the inverse matrix $Z_X^{-1}$.

The magnitude of a finite metric space, defined as above, is expected to behave like an ``effective number of points''. One such expectation is the validity of the inequalities $1 \le \Mag(X) \le \# X$ for $X \neq \emptyset$, where $\# X$ denotes the cardinality of a finite set $X$. Another expectation is that the so-called magnitude function $t \mapsto \Mag(tX)$ is monotonically increasing, where $tX$ denotes the set $X$ equipped with the metric $t d$ obtained by rescaling $d$ by $t > 0$. However, it has long been known that these expectations do not hold generally \cite{L}.

To ensure the expected behaviour, the classical assumption that a finite metric space $X$ is of \textit{negative type} is often made. It is known \cite{M,Sch1} that this assumption is equivalent to the positive definiteness of $Z_{tX}$ for all $t > 0$. Then, the positive definiteness of $Z_X$ implies the inequality $\Mag(X) \ge 1$ for $X \neq \emptyset$, as shown in the original paper \cite{L} in which magnitude was introduced. The other inequality $\Mag(X) \le \# X$ has been proved recently in \cite{AG,GM} under the assumption that $Z_{\frac{1}{2}X}$ is positive definite. Using the technique in \cite{GM}, we can also show that the determinant function $t \mapsto \det Z_{tX}$ is monotonically increasing when $X$ is negative type. (See Appendix \ref{sec:appendix:determinant} for details.) The monotonicity of the magnitude function associated to a negative type metric space had been conjectured since \cite{LM}, but counterexamples were constructed by Venkata Siddharth Pendyala only very recently in personal communication with the first author. His counterexamples involve at least $14$ points. Although the conjecture is false in general, the question of the smallest number of points for which monotonicity can fail remains open.

To date, many properties of the magnitude of negative type finite metric spaces have been proved \textit{without using the triangle inequality}. Consequently, many of these properties can be generalized to negative type \textit{semimetric} spaces, namely, ``metric spaces without the triangle inequality''. Generally, the magnitude can be defined for a matrix with a certain property \cite{L}. We apply this viewpoint to $Z_X$ to define the magnitude $\Mag(X)$ of a semimetric space $X$. In the semimetric case, the negative type assumption is again equivalent to the positive definiteness of $Z_{tX}$ for all $t > 0$. Hence $\Mag(X)$ is also defined as the sum of the entries of $Z_X^{-1}$. Then, in particular, the same proof as in the metric case leads to the inequalities $1 \le \Mag(X) \le \# X$ and the monotonicity of the function $t \mapsto \det Z_{tX}$ for a non-empty semimetric space $X$ of negative type.

In view of the observation above, it is natural to generalize the problem about the monotonicity of the magnitude function to semimetrics:

\begin{problem} \label{problem:main}
Is the magnitude function $t \to \Mag(tX)$ monotonically increasing in $t > 0$ for any finite semimetric space $X$ of negative type?
\end{problem}

Our main result completely resolves Problem \ref{problem:main}.

\begin{thm} \label{thm:main_semimetric}
Let $X$ be a finite semimetric space of negative type.

\begin{itemize}
\item
In the case where $\# X \le 3$, the answer is ``Yes''.

\item
In the case where $\# X \ge 4$, the answer is ``No''.

\end{itemize}
\end{thm}

We next explain these two statements.

Let $X_n$ denote a semimetric space consisting of $n$ points. In the case where $n = 1, 2$, there is no distinction between a metric and a semimetric. Furthermore, $X_1$ and $X_2$ are of negative type. We readily get
\begin{align*}
\Mag(t X_1) &= 1, &
\Mag(t X_2) &= \frac{2}{1 + e^{-t \ell}},
\end{align*}
where $\ell > 0$ is the distance between the two points in $X_2$. Hence the answer to Problem \ref{problem:main} is clearly ``Yes''. (We are using the term ``monotonically increasing'' in the non-strict sense.) In the case where $n = 3$, any $3$-point metric space $X_3$ is of negative type \cite{L}. In this metric case, the monotonicity follows from the fact known as the positivity of the unique weighting on a $3$-point metric space \cite{L}. In contrast, when a negative type semimetric space $X_3$ is not a metric space, the negative type condition for $X_3$ becomes non-trivial. Since the magnitude function $\Mag(tX_3)$ can be explicitly described as a function of $t$, the monotonicity may appear to be easy to prove. However, Problem \ref{problem:main} turns out to be somewhat more difficult than it first appears. Our proof of the monotonicity in the $3$-point case, which will be Theorem \ref{thm:3_point} in \S\ref{sec:3_point_case}, is divided into two parts. In the first part, we use a special property of a $3$-point (semi)metric space to reduce the monotonicity to the positivity of two explicit functions. Then the second part proves the positivity of the functions. The proof is elementary, but is nevertheless non-trivial. 

In the case where $n \ge 4$, there exist $n$-point semimetric spaces $X_n$ of negative type which are not metric spaces and whose magnitude functions are not monotonically increasing. Such examples are explicitly given in the case where $n = 4, 5$. Based on these examples, one can construct an example $X_n$ for each $n \ge 6$. The detailed construction is described in \S\ref{sec:4_point_and_more}. We note that, in these cases of $n \ge 4$, the failure of the monotonicity of $\Mag(t X_n)$ occurs near $t = 0$. This is shown by computing the small-scale limit of the derivative $\lim_{t \to 0}d\Mag(tX_n)/dt$ following the ideas in \cite{RY}.

\smallskip

While Theorem \ref{thm:main_semimetric} completely resolves the monotonicity problem for semimetric spaces of negative type, the smallest number of points for which monotonicity can fail in the metric case remains unknown. As partial results in this direction, we also prove the following in this paper.

\begin{thm} \label{thm:sub_metric}
The following holds true.
\begin{itemize}
\item[(a)]
The magnitude function of any $4$-point metric space is monotonically increasing. 

\item[(b)]
There are $12$-point metric spaces of negative type whose magnitude functions are not monotonically increasing.

\end{itemize}
\end{thm}

Theorem \ref{thm:sub_metric} (a), in contrast to the $4$-point semimetric counterexample in Theorem \ref{thm:main_semimetric}, illustrates the role of the triangle inequality in the monotonicity problem. Consequently, the smallest number of points for which the monotonicity can fail among metric spaces of negative type is at least $5$ and at most $12$.

\bigskip

The organization of this paper is as follows. In \S\ref{sec:preliminary}, we review some basic facts about semimetric spaces and their magnitude to be used in this paper. In \S\ref{sec:3_point_case}, the monotonicity in the $3$-point case (Theorem \ref{thm:3_point}) is proved. \S\ref{sec:4_point_and_more} is devoted to the constructions of negative type semimetric spaces consisting of $n \ge 4$ points whose magnitude functions are not monotonically increasing near $t = 0$. In \S\ref{sec:metric}, we prove our results about $4$- and $12$-point metric spaces as stated in Theorem \ref{thm:sub_metric} (which will be Theorem \ref{thm:monotonicity_4_point_metric_space} and Theorem \ref{thm:12_point_example}.) Finally, in Appendix \ref{sec:appendix:determinant}, we prove the monotonicity of the determinant function $t \mapsto \det Z_{tX}$ for any finite semimetric space $X$ of negative type.

\bigskip

\begin{acknowledgement}
KG thanks Shuhei Ohyama and Venkata Siddharth Pendyala for valuable communications.
\end{acknowledgement}

\paragraph*{{\it AI use statement.}}
GPT-5.6 Sol was used to generate ideas for the proof of Theorem \ref{thm:3_point} in response to prompts from the second-named author. The first-named author used GPT-6 Pro to assist in developing the proof of Theorem \ref{thm:sub_metric}. 
Both authors verified the final argument and computations and take full responsibility for the mathematical content of this paper.


\section{Preliminaries}
\label{sec:preliminary}

We review here some basic facts about semimetric spaces and their magnitude to be used in this paper.

\subsection{Semimetric spaces of negative type}

Let $X$ be a (finite) set. As is well-known, a map $d : X \times X \to [0, \infty)$ is said to be a \textit{metric} on $X$ when the following (i), (ii) and (iii) hold.

\begin{enumerate}
\item[(i)]
$d_{x, y} = 0$ if and only if $x = y \in X$. 

\item[(ii)]
$d_{x, y} = d_{y, x}$ for all $x, y \in X$.

\item[(iii)]
$d_{x, y} + d_{y, z} \ge d_{x, z}$ for all $x, y, z \in X$.

\end{enumerate}

We say $d$ is a \textit{semimetric} on $X$ when (i) and (ii) are satisfied. That is, a semimetric is a ``metric without the triangle inequality''. In any event, we call $d$ the distance function on a (semi)metric space $X$. We often omit $d$ and just write $X$ to mean a (semi)metric space $(X, d)$.

\medskip

We then recall the classical definition of \textit{negative type}. We refer to \cite{DL} for more information. To state the definitions, let $M(n, \R)$ denote the space of real $n \times n$ matrices, $\langle \ , \ \rangle$ the standard inner product on $\R^n$ and $\1 \in \R^n$ the vector whose entries are $1$. We define a linear subspace $\Pi_0 \subset \R^n$ by $\Pi_0 = \{ v \in \R^n |\ \langle v, \1 \rangle = 0 \}$. Given a semimetric $d$ on a finite set $X = \{ 0, 1, \ldots, n -1 \}$ consisting of $n > 1$ points, we define the \textit{distance matrix} $D \in M(n, \R)$ of $(X, d)$ by $D = (d_{i, j})_{i, j \in X}$.

\begin{dfn}
Let $d$ be a (semi)metric on a finite set $X$ consisting of $n \ge 1$ points. We say that $(X, d)$ (or $d$) is of \textit{negative type} when its distance matrix $D$ is conditionally negative semidefinite. That is, $\langle v, Dv \rangle \le 0$ for all $v \in \Pi_0$.
\end{dfn}

The definition of being negative type has a number of equivalent formulations. We here state one of them which makes use of \textit{excess matrices}. For a (semi)metric $d$ on a finite set $X = \{ 0, 1, \ldots, n - 1 \}$ consisting of $n > 1$ points and $k \in X$, we define the $k$-th \textit{excess matrix} $E^{(k)} \in M(n - 1, \R)$ by $E^{(k)} = ( d_{i, k} + d_{j, k} - d_{i, j} )_{i, j \in X \backslash \{ k \}}$.

\begin{lem} \label{lem:neg_type_and_excess}
Let $(X, d)$ be a finite (semi)metric space consisting of $n > 1$ points.
\begin{itemize}
\item[(a)]
All the excess matrices $E^{(k)}$, ($k \in X$) have the same inertia.

\item[(b)]
$(X, d)$ is of negative type if and only if $E^{(k)}$ is positive semidefinite.

\end{itemize}
\end{lem}

\begin{proof}
(a)
For a semimetric $d$ on $X = \{ 0, 1, \ldots, n - 1 \}$, one has the following expressions of the distance matrix and the $0$-th excess matrix
\begin{align*}
D &=
\left(
\begin{array}{cc}
0 & {}^td^{(0)} \\
d^{(0)} & D^{(0)}
\end{array}
\right), 
&
E^{(0)}
&=
\1 {}^t d^{(0)} + d^{(0)} {}^t \1 - D^{(0)},
\end{align*}
where ${}^t$ means the transpose. Let us introduce a matrix $\mathrm{CM}$ of Cayley-Menger type and an invertible matrix $P^{(0)}$, both belonging to $M(n+1, \R)$, as follows:
\begin{align*}
\mathrm{CM} 
&=
\left(
\begin{array}{cc}
0 & {}^t\1 \\
\1 & D
\end{array}
\right)
=
\left(
\begin{array}{ccc}
0 & 1 & {}^t \1 \\
1 & 0 & {}^t d^{(0)} \\
\1 & d^{(0)} & D^{(0)} 
\end{array}
\right),
&
P^{(0)}
&=
\left(
\begin{array}{ccc}
1 & 0 & - {}^td^{(0)} \\
0 & 1 & - {}^t\1 \\
0 & 0 & I
\end{array}
\right),
\end{align*}
where $I \in M(n-1, \R)$ is the identity matrix. We have
$$
{}^tP^{(0)} (\mathrm{CM}) P^{(0)}
=
\left(
\begin{array}{ccc}
0 & 1 & 0 \\
1 & 0 & 0 \\
0 & 0 & -E^{(0)}
\end{array}
\right).
$$
Thus, if $(n_+, n_-, n_0)$ and $(N_+, N_-, N_0)$ respectively denote the inertias of $E^{(0)}$ and $\mathrm{CM}$, then they are related by $(N_+, N_-, N_0) = (n_- + 1, n_+ + 1, n_0)$. For each $k$, the inertia of $E^{(k)}$ is similarly related to that of $\mathrm{CM}$. This proves that the inertia of $E^{(k)}$ is independent of $k \in X$.

(b)
A vector $v \in \Pi_0 \subset \R^n$ admits a unique expression
$$
v = 
\left(
\begin{array}{c}
- \langle u, \1 \rangle \\
u
\end{array}
\right)
$$
in terms of $u \in \R^{n-1}$. Now the formula
$$
- \langle v, D v \rangle
=
2 \langle u, \1 \rangle \langle u, d^{(0)} \rangle
- \langle u, D^{(0)} u \rangle
=
\langle u, E^{(0)} u \rangle
$$
implies the equivalence. 
\end{proof}

For later use, we describe the so-called ``Schoenberg embedding'' or ``quadratic embedding'' of a finite (semi)metric space of negative type.

\begin{prop}[\cite{Sch2}] \label{prop:Schoenberg_embedding}
Let $d$ be a semimetric on a finite set $X = \{ 0, 1, \ldots, n - 1 \}$ consisting of $n > 1$ points. Suppose that $(X, d)$ is negative type and the inertia of the $0$-th excess matrix $E^{(0)} \in M(n-1, \R)$ is $(p, 0, n - 1 - p)$, namely $E^{(0)}$ has $p$ positive eigenvalues counted with multiplicity. Then, $p \ge 1$, and there are vectors $v_1, \ldots, v_{n-1} \in \R^p$ with the following properties.
\begin{itemize}
\item[(a)]
$\{ v_1, \ldots, v_p \}$ is a basis of $\R^p$.

\item[(b)]
If we put $v_0 = 0$, then $d_{i, j} = \lVert v_i - v_j \rVert^2$ for any $i, j \in X$.

\item[(c)]
$E^{(0)}/2$ agrees with the Gram matrix of the vectors $v_1, \ldots, v_{n-1}$
$$
\frac{1}{2}E^{(0)} 
= \mathrm{Gram}(v_1, \ldots, v_{n-1}) 
= ( \langle v_i, v_j \rangle )_{i, j = 1}^{n-1}.
$$

\end{itemize}
\end{prop}

Note that the description above is not, in general, an isometric embedding of $(X, d)$ into $\R^p$ with the standard $L^2$-metric. Note also that (b) and (c) above are equivalent properties. The following characterization of the triangle inequality in terms of the Schoenberg embedding can be found, for example, in \cite{NW}.

\begin{lem} \label{lem:metric_condition}
In the description of a semimetric $(X, d)$ in Proposition \ref{prop:Schoenberg_embedding}, the triangle inequality $d_{i, j} + d_{j, k} \ge d_{i, k}$ is equivalent to $\langle v_i - v_j, v_k - v_j \rangle \ge 0$. In other words, $(X, d)$ gives rise to a metric space if and only if $\{ v_0, \ldots, v_{n-1} \} \subset \R^p$ is a \textit{non-obtuse set}.
\end{lem}

\begin{proof}
We can directly see $d_{i, j} + d_{j, k} - d_{i, k} = 2 \langle v_i - v_j, v_k - v_j \rangle$.
\end{proof}

\begin{rem} \label{rem:non_obtuse}
A classical result of Danzer and Gr\"{u}nbaum \cite{DG} state that a non-obtuse set in $\R^d$ contains at most $2^d$ points. It follows that the possible inertia of an excess matrix of an $n$-point metric space of negative type is $(p, 0, n - 1 - p)$, where $p$ runs from $p_-$ to $n - 1$, and $p_-$ is given by $p_- = \min \{ p |\ n \le 2^p \} = \lceil \log_2 n \rceil$.
\end{rem}

\subsection{Magnitude of semimetric spaces}

Here we introduce the magnitude of a semimetric space as a straightforward generalization of that of a metric space \cite{L}.

\begin{dfn}
Let $X = (X, d)$ be a semimetric space consisting of $n \ge 1$ points.

\begin{enumerate}
\item[(a)]
The \textit{similarity matrix} (or \textit{zeta matrix}) of $X$ is defined as the matrix $Z_X = ( e^{-d_{i, j}} )_{i, j \in X} \in M(n, \R)$ whose $(i, j)$-entry is $\exp(-d_{i, j})$.

\item[(b)]
A \textit{weighting} on $X$ is a vector $w = (w_i)_{i \in X} \in \R^n$ such that $Z_X w = \1$.

\item[(c)]
When $X$ admits a weighting $w$, the \textit{magnitude} of $X$ is defined as the sum of the entries of $w$, which we describe as
$$
\Mag(X) = \langle \1, w \rangle.
$$

\item[(d)]
The \textit{magnitude function} is a function $t \mapsto \Mag(tX)$ defined for positive real numbers $t$ such that the magnitude of $tX = (X, td)$ is defined, where $t d$ is the rescaled semimetric $(t d)_{i, j} = t \cdot d_{i, j}$.

\end{enumerate}
\end{dfn}

It is easy to see that the magnitude above is independent of the choice of a weighting. If $Z$ is invertible, then the weighting is unique and is given by $w = Z^{-1} \1$, so that the magnitude is the sum of the entries of $Z^{-1}$, which we describe as
$$
\Mag(X) = \langle \1, Z_X^{-1} \1 \rangle.
$$

The following characterization of negative type will be important for our purposes.

\begin{lem}
A finite semimetric space $X$ is of negative type if and only if $Z_{tX}$ is positive definite for all $t > 0$.
\end{lem}

\begin{proof}
In the case where $X$ is a metric space, Meckes \cite{M} proved that the negative type assumption on $X$ implies the positive definiteness of $Z_{tX}$ for $t > 0$ by using the theory of ``positive definite functions''  and the ``Schoenberg embedding'' in Proposition \ref{prop:Schoenberg_embedding}. Since this proposition is valid for semimetrics, so is Meckes' argument. The converse follows from the following expansion at $t = 0$, as in Schoenberg's argument in \cite{Sch1}
$$
\langle v, Z_{tX} v \rangle
=
- t \langle v, D v \rangle 
+ \frac{t^2}{2} \langle v, D^{\circ 2}v \rangle + \cdots,
$$
where $v \in \Pi_0$ and $D^{\circ 2} = (d_{i, j}^2)_{i, j}$.
\end{proof}

Note that a metric space $X$ such that $Z_{tX}$ is positive definite for all $t > 0$ is called \textit{stably positive definite} \cite{L,M}.

\bigskip

For later use, we give a formula of the derivative of the magnitude function.

\begin{lem} \label{lem:derivative_general}
Let $d$ be a semimetric on a finite set $X = \{ 0, 1, \ldots, n  - 1\}$ consisting of $n \ge 1$ points. Suppose that $Z_{tX}$ is invertible at $t > 0$. Then the derivative of the magnitude function $\Mag(tX)$ at $t$ is
$$
\frac{d\Mag(tX)}{dt}
=
\langle Z_{tX}^{-1}\1, (D_X \circ Z_{tX}) Z_{tX}^{-1}\1 \rangle
=
\langle w(t), (D_X \circ Z_{tX}) w(t) \rangle,
$$
where $\circ$ denotes the Hadamard (entrywise) product of matrices, and $w(t) = Z_{tX}^{-1} \1$ is the weighting of $tX$. If we express the entries of $w(t)$ as $w_i(t)$ and put $q = e^{-t}$, then we also have
$$
\frac{d\Mag(tX)}{dt}
=
2 \sum_{i < j} d_{i, j} q^{d_{i, j}} w_i w_j.
$$
\end{lem}

\begin{proof}
Taking the derivative of the identity $Z_{tX} Z_{tX}^{-1} = I$, we have
$$
dZ_{tX}^{-1}/dt = -Z_{tX}^{-1} (dZ_{tX}/dt) Z_{tX}^{-1}.
$$
Now the first formula follows from $dZ_{tX}/dt = - D_X \circ Z_{tX}$. Then an explicit computation proves the second formula.
\end{proof}

\section{The case of $3$ points}
\label{sec:3_point_case}

The goal of this section is to prove the monotonicity in the $3$-point case.

\begin{thm} \label{thm:3_point}
Let $X = (X, d)$ be a $3$-point semimetric space of negative type. Then we have $d\Mag(tX)/dt > 0$ for all $t > 0$, and hence the magnitude function is strictly monotonically increasing for all $t > 0$.
\end{thm}

The proof is divided into two parts: The first part reduces our monotonicity problem to certain inequalities. Then the second part proves the resulting inequalities (as Proposition \ref{prop:3_point}).

\subsection{The first part of the proof}

We are to prove that $d\Mag(tX)/dt > 0$ for all $t > 0$ if $d$ is a semimetric on a $3$-point set $X = \{ 0, 1, 2 \}$ of negative type. Under this assumption, $Z_{tX}$ is positive definite. Let $w = w(t) = Z_{tX}^{-1} \1$ be the unique weighting. Let $u_0, u_1, u_2$ be 
\begin{align*}
u_0 &= 1 - q^{d_{0,1}} - q^{d_{0,2}} + q^{d_{1,2}}, \\
u_1 &= 1 - q^{d_{0,1}} + q^{d_{0,2}} - q^{d_{1,2}}, \\
u_2 &= 1 + q^{d_{0,1}} - q^{d_{0,2}} - q^{d_{1,2}},
\end{align*}
where $q = e^{-t}$. Then each entry $w_i$ of $w = (w_i)$ is expressed as
\begin{align*}
w_0 &= \frac{(1 - q^{d_{1,2}})u_0}{\det Z_{tX}}, &
w_1 &= \frac{(1 - q^{d_{0,2}})u_1}{\det Z_{tX}}, &
w_2 &= \frac{(1 - q^{d_{0,1}})u_2}{\det Z_{tX}}.
\end{align*}
It is easy to see that the triangle inequality $d_{0,1} + d_{1,2} \ge d_{0,2}$ implies
$$
u_1 
\ge 1 - q^{d_{0,1}} + q^{d_{0,1} + d_{1,2}} - q^{d_{1,2}}
= (1 - q^{d_{0,1}})(1 - q^{d_{1,2}}) > 0.
$$
Thus, if $d$ is a metric on a $3$-point set, then the entries of the weighting are always positive, as shown in \cite{L}. Now, from the formula for the derivative in Lemma \ref{lem:derivative_general}, it immediately follows that $d\Mag(tX)/dt > 0$ provided that $d$ is a metric.

\medskip

Henceforth we consider the case where $d$ is not a metric. To address this case, we use a special property of a $3$-point semimetric: We can express $d_{i, j}$ for distinct $i, j \in X = \{ 0, 1, 2 \}$ as
\begin{align*}
d_{0, 1} &= a_0 + a_1, &
d_{0, 2} &= a_0 + a_2, &
d_{1, 2} &= a_1 + a_2,
\end{align*}
in which $a_0, a_1, a_2$ are uniquely determined by
\begin{align*}
a_0 &= \frac{d_{0, 1} + d_{0, 2} - d_{1, 2}}{2}, &
a_1 &= \frac{d_{0, 1} + d_{1, 2} - d_{0, 2}}{2}, &
a_2 &= \frac{d_{0, 2} + d_{1, 2} - d_{0, 1}}{2}.
\end{align*}
Suppose the failure of a triangle inequality $a_0 < 0$ for example. Then the semimetric conditions $d_{0, 1} > 0$ and $d_{0, 2} > 0$ imply $a_1 > -a_0$ and $a_2 > -a_0$. Furthermore, we can normalize $a_0$, because the monotonicity of $\Mag(tX)$ is a property which is invariant under the rescaling of the distance. In summary, we can assume without loss of generality that 
$$
a_0 = -1 < 0 < a_1 \le a_2.
$$
To simplify the notation, we introduce new variables $0 < \alpha \le \beta$ to express
\begin{align*}
a_1 &= \alpha + 1, &
a_2 &= \beta + 1.
\end{align*}
Note that 
\begin{align*}
d_{0, 1} &= \alpha, &
d_{0, 2} &= \beta, &
d_{1, 2} &= \alpha + \beta + 2.
\end{align*}

\begin{lem}
Suppose that a semimetric $d$ on $X = \{ 0, 1, 2 \}$ is expressed as above in terms of $a_0, a_1, a_2$. Then $d$ is of negative type if and only if 
$$
a_0a_1 + a_0a_2 + a_1a_2 \ge 0.
$$
In the case where $a_0 = -1$, $a_1 = \alpha + 1$ and $a_2 = \beta + 1$, the semimetric $d$ is of negative type if and only if
$$
\alpha \beta \ge 1.
$$
\end{lem}

\begin{proof}
It suffices to see the condition for an excess matrix to be positive semidefinite. The $0$-th excess matrix becomes
$$
E^{(0)} =
2
\left(
\begin{array}{cc}
a_0 + a_1 & a_0 \\
a_0 & a_0 + a_2
\end{array}
\right).
$$
Since $d$ is a semimetric, the diagonal entries of $E^{(0)}$ are positive. Then, $E^{(0)}$ is positive semidefinite if and only if 
$$
\det (E^{(0)}/2)
= 
(a_0 + a_1)(a_0 + a_2) - a_0^2
=
a_0a_1 + a_0a_2 + a_1a_2 \ge 0.
$$
A direct calculation shows the remaining claim.
\end{proof}

Substituting $d_{0, 1} = \alpha$, $d_{0, 2} = \beta$ and $d_{1, 2} = \alpha + \beta + 2$ into $u_i$, we get
\begin{align*}
u_0 &= 1 - q^\alpha - q^\beta + q^{\alpha + \beta + 2}, \\
u_1 &= 1 - q^\alpha + q^\beta - q^{\alpha + \beta + 2}, \\
u_2 &= 1 + q^\alpha - q^\beta - q^{\alpha + \beta + 2}.
\end{align*}
From $0 < \alpha \le \beta$, it follows that $0 < u_1 \le u_2$. Now we have
\begin{align*}
&
\frac{(\det Z_{tX})^2}{2} \frac{d\Mag(tX)}{dt}
\\
&=
(\det Z_{tX})^2 (
d_{0, 1}q^{d_{0, 1}} w_0w_1 +
d_{0, 2}q^{d_{0, 2}} w_0w_2 +
d_{1, 2}q^{d_{1, 2}} w_1w_2
)
\\
&=
(\det Z_{tX})^2 (
\alpha q^\alpha w_0 w_1 +
\beta q^\beta w_0 w_2 +
(\alpha + \beta + 2) q^{\alpha + \beta + 2} w_1 w_2
)
\\
&=
(\det Z_{tX})^2 \big(
q^\alpha w_1 \big( \alpha w_0 + (\alpha + 1)q^{\beta + 2} w_2 \big)
+
q^\beta w_2 \big( \beta w_0 + (\beta + 1) q^{\alpha + 2} w_1 \big)
\big)
\\
&=
q^\alpha (1 - q^\beta) u_1
\big(
\alpha (1 - q^{\alpha + \beta + 2})u_0
+ (\alpha + 1) q^{\beta + 2} (1 - q^\alpha) u_2
\big)
\\
&\quad
+
q^\beta (1 - q^\alpha) u_2 
\big(
\beta (1 - q^{\alpha + \beta + 2})u_0 
+ (\beta + 1) q^{\alpha + 2} (1 - q^\beta) u_1
\big)
\\
&=
q^\alpha (1 - q^\beta) u_1  \phi_{q, \alpha}(q^\alpha, q^\beta)
+
q^\beta (1 - q^\alpha) u_2  \phi_{q, \beta}(q^\beta, q^\alpha),
\end{align*}
where $\phi_{q, \alpha}(x, z)$ is defined by
\begin{multline*}
\phi_{q, \alpha}(x, z)
\\
=
\alpha (1 - q^2 x z)(1 - x - z + q^2 xz)
+ 
(\alpha + 1) q^2 z (1 - x)(1 + x - z - q^2xz).
\end{multline*}
Thus, in order to get $d\Mag(tX)/dt > 0$, it suffices to show $\phi_{q,\alpha}(q^\alpha, q^\beta) > 0$ and $\phi_{q,\beta}(q^\beta, q^\alpha) > 0$.

\subsection{The second part of the proof}

By the preceding argument, we can complete the proof of Theorem \ref{thm:3_point} by proving the following proposition.

\begin{prop} \label{prop:3_point}
Let $\alpha$ and $\beta$ be real numbers such that $0 < \alpha \le \beta$ and $\alpha \beta \ge 1$, and let $q \in (0, 1)$.
\begin{enumerate}
\item[(a)]
We have $\phi_{q,\alpha}(q^\alpha, q^\beta) > 0$. 

\item[(b)]
We have $\phi_{q,\beta}(q^\beta, q^\alpha) > 0$.

\end{enumerate}
\end{prop}

To prove the proposition, we first establish several lemmas.

\begin{lem} \label{lem:3_point_phi_decreasing}
If $q \in (0, 1)$ and $\alpha > 0$, then the quadratic polynomial $\phi_{q, \alpha}(q^\alpha, z)$ in $z > 0$ is monotonically decreasing.
\end{lem}

\begin{proof}
We have
\begin{align*}
\frac{\partial \phi_{q, \alpha}(q^\alpha, z)}{\partial z} 
&=
- \alpha + (1 + \alpha)q^2 - q^{2 + 2\alpha}
\\
&\quad
+ 2q^2
\big(
 -1 - \alpha 
+ (1 + 2\alpha)q^\alpha - (1 + \alpha)q^{2 + \alpha} + q^{2 + 2\alpha}
\big) z.
\end{align*}
Hence the lemma will be established by showing that the following functions $\xi_1(q)$ and $\xi_2(q)$ are non-positive
\begin{align*}
\xi_1(q) &= - \alpha + (1 + \alpha) q^2 - q^{2 + 2\alpha}, \\
\xi_2(q) &=  -1 - \alpha 
+ (1 + 2\alpha)q^\alpha - (1 + \alpha)q^{2 + \alpha} + q^{2 + 2\alpha}.
\end{align*}
For $\xi_1(q)$, its derivative is 
$$
\xi'_1(q) = 2(1 + \alpha)q(1 - q^{2\alpha}) > 0.
$$
Hence we find $\xi_1(q) \le \xi_1(1) = 0$. For $\xi_2(q)$, we have
\begin{align*}
\xi_2'(q) &= q^{\alpha - 1}
\big(
\alpha(1 + 2 \alpha) -(1 + \alpha)(2 + \alpha)q^2 + 
2 (1 + \alpha) q^{2 + \alpha}
\big).
\end{align*}
If we put $\xi_3(q) = \xi_2'(q)q^{1 - \alpha}$, then 
$$
\xi_3'(q) = -2 (1 + \alpha)(2 + \alpha)q(1 - q^\alpha) < 0.
$$
From $\xi_3(q) \ge \xi_3(1) = \alpha^2$, it follows that $\xi_2'(q) > 0$ and $\xi_2(q) \le \xi_2(1) = 0$.
\end{proof}

\begin{lem} \label{lem:3_point_estimate_at_diagonal}
For $q \in (0, 1)$ and $\alpha > 0$, we define $G_\alpha(q)$ by
$$
G_\alpha(q) = \phi_{q, \alpha}(q^\alpha, q^\alpha)/(1 - q^{2 + 2\alpha})
=
\alpha - 2\alpha q^\alpha + (\alpha + 1)q^{\alpha + 2} - q^{2 \alpha + 2}.
$$
If $\alpha \ge 1$, then $G_\alpha(q) > 0$.
\end{lem}

\begin{proof}
We compute the derivative of $G_\alpha(q)$ to get
$$
G'_\alpha(q) =
q^{\alpha - 1}
\big(
-2 \alpha^2 + (1 + \alpha)(2 + \alpha) q^2 - 2 (1 + \alpha) q^{\alpha+2}
\big).
$$
If we put $g_\alpha(q) = G'_\alpha(q)q^{1 - \alpha}$, then 
$$
g'_\alpha(q)
=
2(1 + \alpha)(2 + \alpha)q(1 - q^\alpha) > 0.
$$
Hence $g_\alpha(q) < g_\alpha(1) = \alpha(1 - \alpha) \le 0$ by $\alpha \ge 1$. As a result, we find that $G'_\alpha(q) < 0$ and $G_\alpha(q) > G_\alpha(1) = 0$.
\end{proof}

\begin{lem} \label{lem:3_point_estimate_at_boundary}
For $q \in (0, 1)$ and $\alpha > 0$, we define $H_\alpha(q)$ by
\begin{align*}
H_\alpha(q) 
&= 
\phi_{q, \alpha}(q^\alpha, q^{1/\alpha})
\\
&=
\alpha (1 - q^{2 + \alpha + \frac{1}{\alpha}})
(1 - q^\alpha - q^{\frac{1}{\alpha}} + q^{2 + \alpha + \frac{1}{\alpha}})
\\
&\quad
+ 
(\alpha + 1) q^{2 + \frac{1}{\alpha}} (1 - q^\alpha)
(1 + q^\alpha - q^{\frac{1}{\alpha}} - q^{2 + \alpha + \frac{1}{\alpha}}).
\end{align*}
Then $H_\alpha(q) > 0$.
\end{lem}

\begin{proof}
We prove the inequality case by case:
\begin{enumerate}
\item[(i)]
The case where $0 < \alpha < 1$; and

\item[(ii)]
The case where $\alpha \ge 1$.

\end{enumerate}

(i) 
In the case that $0 < \alpha < 1$, returning to the notation by using $t > 0$ from $q = e^{-t}$, we can verify the following formula
\begin{align*} 
&\frac{1}{2} e^{(\alpha+\frac{1}{\alpha}+2)t}H_\alpha(e^{-t})
\\
&=
\left( 
\sinh {\textstyle ((\alpha+1)t) }
- 
(\alpha+1) \sinh t  
\right) 
\\
&
\quad
\times
\left( 
(\alpha+1) 
\sinh {\textstyle \left(\left(\frac1\alpha +1 \right)t\right) }
- 
\cosh {\textstyle \left(\left(\frac1\alpha +1 \right)t\right) }
+ 
\cosh {\textstyle ((\alpha+1)t) }
\right) 
\\
&
\quad\quad
+
\left( 
\alpha 
\sinh {\textstyle \left(\left(\frac1\alpha +1 \right)t\right) }
- 
\sinh ((\alpha+1)t) 
\right)
\left( 
\alpha \sinh t 
+ 
\cosh ((\alpha+1)t) 
- 
\cosh t  
\right). 
\end{align*}
We then observe that this expression is the sum of products of positive terms: We know that the function $x \to \frac{\sinh x}{x}$ is strictly monotonically increasing for $x > 0$. Thus, if $0 < \alpha < 1$ and $t > 0$, then we get
\begin{align*}
\sinh((\alpha + 1)t) - (\alpha + 1) \sinh t &> 0, \\
\alpha \sinh {\textstyle \left( \left( \frac{1}{\alpha} + 1 \right)t \right) }
- \sinh ((\alpha + 1)t) & > 0.
\end{align*}
For $t > 0$, we also have
\begin{align*}
&
(\alpha+1) 
\sinh {\textstyle \left(\left(\frac1\alpha +1 \right)t\right) }
- 
\cosh {\textstyle \left(\left(\frac1\alpha +1 \right)t\right) }
+ 
\cosh {\textstyle ((\alpha+1)t) }
\\
&=
\frac{1}{2} \alpha e^{\left( \frac{1}{\alpha} + 1 \right)t}
-
\frac{1}{2}(\alpha + 2) e^{- \left( \frac{1}{\alpha} + 1 \right)t}
+
\cosh ((\alpha + 1)t)
\\
&>
\frac{1}{2}\alpha - \frac{1}{2}(\alpha + 2) + \cosh((\alpha + 1)t)
=
\cosh((\alpha + 1)t) - 1 > 0.
\end{align*}
Since $\cosh x$ is strictly monotonically increasing in $x > 0$, we have
$$
\alpha \sinh t 
+ 
\cosh ((\alpha+1)t) 
- 
\cosh t
> 0.
$$
Hence we conclude $H_\alpha(q) > 0$ for all $0 < q < 1$ and $0 < \alpha < 1$.

(ii)
In the case where $\alpha \ge 1$, we use another formula 
\begin{align*}
&\frac{1}{2} e^{(\alpha+\frac{1}{\alpha}+2)t}
H_\alpha(e^{-t})
\\
&=
\alpha 
\left( 
\sinh {\textstyle \left(\left(\frac1\alpha +1 \right)t\right) }
- 
\left(  \frac{1}{\alpha} +1 \right) \sinh t  
\right) 
\\
&
\quad
\times 
\left( 
(\alpha+1) 
\sinh {\textstyle \left(\left(\frac1\alpha +1 \right)t\right) }
- 
\cosh {\textstyle \left(\left(\frac1\alpha +1 \right)t\right) }
+ 
\cosh ((\alpha+1)t) 
\right) 
\\
&
\quad\quad
+ 
\left( 
\sinh ((\alpha+1)t) 
- 
\alpha 
\sinh {\textstyle \left(\left(\frac1\alpha +1 \right)t\right) }
\right) 
\\
& 
\quad\quad\quad
\times
\left( 
(\alpha+1) 
\sinh {\textstyle \left(\left(\frac1\alpha +1 \right)t\right) }
- 
\cosh {\textstyle \left(\left(\frac1\alpha +1 \right)t\right) }
- 
\alpha 
\sinh t 
+ 
\cosh t 
\right).  
\end{align*}
We then verify that this expression is the sum of a positive term and a non-negative term: Because $\frac{\sinh x}{x}$ is strictly monotonically increasing in $x > 0$, we have
\begin{align*}
\sinh {\textstyle \left(\left(\frac1\alpha +1 \right)t\right) }
- 
\left(  \frac{1}{\alpha} +1 \right) \sinh t  
&>
0,
\\
\sinh ((\alpha+1)t) 
- 
\alpha 
\sinh {\textstyle \left(\left(\frac1\alpha +1 \right)t\right) }
&\ge
0,
\end{align*}
provided that $\alpha \ge 1$. Since $\sinh x$ and $\cosh x$ are strictly monotonically increasing in $x > 0$, we use $\alpha \ge 1$ to get
$$
(\alpha+1) 
\sinh {\textstyle \left(\left(\frac1\alpha +1 \right)t\right) }
- 
\cosh {\textstyle \left(\left(\frac1\alpha +1 \right)t\right) }
+ 
\cosh ((\alpha+1)t)
>
0. 
$$
Finally, because $\sinh x$ is monotonically increasing in $x > 0$, we have
\begin{align*}
&
(\alpha+1) 
\sinh {\textstyle \left(\left(\frac1\alpha +1 \right)t\right) }
- 
\cosh {\textstyle \left(\left(\frac1\alpha +1 \right)t\right) }
- 
\alpha 
\sinh t 
+ 
\cosh t 
\\
&=
\alpha
\left(
\sinh {\textstyle \left(\left(\frac1\alpha +1 \right)t\right) }
-
\sinh t
\right)
- e^{ - \left( \frac{1}{\alpha} + 1 \right)t}
+ \cosh t
\\
&>
\alpha
\left(
\sinh {\textstyle \left(\left(\frac1\alpha +1 \right)t\right) }
-
\sinh t
\right)
- 1
+ \cosh t
> 0.
\end{align*}
Hence $H_\alpha(q) > 0$ for $q \in (0, 1)$ and $\alpha \ge 1$.
\end{proof}

We are now in a position to prove Proposition \ref{prop:3_point}.

\begin{proof}[Proof of Proposition \ref{prop:3_point}]
To prove (a), we think of $\phi_{q, \alpha}(q^\alpha, q^\beta)$ as $\phi_{q, \alpha}(q^\alpha, z)$ with $z = q^\beta$ substituted. By Lemma \ref{lem:3_point_phi_decreasing}, the function $\phi_{q, \alpha}(q^\alpha, z)$ is monotonically decreasing in $z$. Since $\beta$ is subject to $0 < \alpha \le \beta$ and $\alpha \beta \ge 1$, we have
$$
\beta \ge \max\bigg\{ \alpha, \frac{1}{\alpha} \bigg\}
=
\left\{
\begin{array}{ll}
\alpha, & (\alpha \ge 1) \\
\frac{1}{\alpha}. & (0 < \alpha < 1)
\end{array}
\right.
$$
Now (a) is proved by the following estimates of $\phi_{q, \alpha}(q^\alpha, q^\beta)$.
\begin{itemize}
\item
In the case that $\alpha \ge 1$, we have $q^\beta \le q^\alpha$. Hence Lemma \ref{lem:3_point_estimate_at_diagonal} implies
$$
\phi_{q, \alpha}(q^\alpha, q^\beta)
\ge
\phi_{q, \alpha}(q^\alpha, q^\alpha)
=
(1 - q^{2 + 2\alpha}) G_\alpha(q) > 0.
$$

\item
In the case that $0 < \alpha < 1$, we have $q^\beta \le q^{\frac{1}{\alpha}}$. Hence Lemma \ref{lem:3_point_estimate_at_boundary} implies
$$
\phi_{q, \alpha}(q^\alpha, q^\beta)
\ge
\phi_{q, \alpha}(q^\alpha, q^{\frac{1}{\alpha}})
=
H_\alpha(q) > 0.
$$

\end{itemize}

To prove (b), we think of $\phi_{q, \beta}(q^\beta, q^\alpha)$ as $\phi_{q, \beta}(q^\beta, x)$ with $x = q^\alpha$ substituted. By Lemma \ref{lem:3_point_phi_decreasing}, the function $\phi_{q, \beta}(q^\beta, x)$ is monotonically decreasing in $x$. Since $\alpha \ge \frac{1}{\beta}$, we get $q^\alpha \le q^{\frac{1}{\beta}}$. Therefore we conclude 
$$
\phi_{q, \beta}(q^\beta, q^\alpha) 
\ge 
\phi_{q, \beta}(q^\beta, q^{\frac{1}{\beta}})
=
H_\beta(q) > 0
$$
by using Lemma \ref{lem:3_point_estimate_at_boundary}.
\end{proof}

\section{The case of $4$ or more points}
\label{sec:4_point_and_more}

This section presents examples of negative type semimetrics $d$ on an $n$-point set $X_n$ with $n \ge 4$ whose magnitude functions are not monotonically increasing.

\subsection{Examples of $4$-point semimetric spaces}
\label{subsec:4_point_example}

In view of Proposition \ref{prop:Schoenberg_embedding}, we construct a semimetric $d$ on a $4$-point set $X_4 = \{ 0, 1, 2, 3 \}$ by taking $4$ points in a Euclidean space. Let us consider the following vectors $v_0, v_1, v_2, v_3$ in $\R^2$
\begin{align*}
v_0
&=
\left(
\begin{array}{c}
0 \\ 0
\end{array}
\right),
&
v_1 
&=
\left(
\begin{array}{c}
1 \\ 0
\end{array}
\right),
&
v_2 
&=
\left(
\begin{array}{c}
2 \\ s
\end{array}
\right),
&
v_3 
&=
\left(
\begin{array}{c}
4 \\ 4s
\end{array}
\right),
\end{align*}
where $s \in \R$ is a parameter. We then put $d_{i, j} = \lVert v_i - v_j \rVert^2$ for $i, j \in X_4$. This $d$ defines a semimetric on $X_4$ for any $s$. The distance matrix is calculated as follows
$$
D =
\left(
\begin{array}{cccc}
0 & 1 & 4 + s^2 & 16(1 + s^2) \\
1 & 0 & 1 + s^2 & 9 + 16s^2 \\
4 + s^2 & 1 + s^2 & 0 & 4 + 9s^2 \\
16(1 + s^2) & 9 + 16s^2 & 4 + 9s^2 & 0
\end{array}
\right).
$$
For $d$ on $X_4$ to be a metric, $\{ v_0, v_1, v_2, v_3 \} \subset \R^2$ must form a non-obtuse set by Lemma \ref{lem:metric_condition}. However, the angle $\angle v_0 v_1 v_2$ is obtuse, and $d$ on $X_4$ is not a metric. We indeed have
$$
d_{0, 1} + d_{1, 2} - d_{0, 2}
=
1 + (1 + s^2) - (4 + s^2)
=
-2 < 0.
$$
Since the vectors $v_1, v_2, v_3$ are in $\R^2$ with the standard positive definite inner product, the $0$-th excess matrix $E^{(0)}$ 
$$
E^{(0)}
=
2 \mathrm{Gram}(v_1, v_2, v_3)
=
2 (\langle v_i, v_j \rangle)_{i, j = 1}^3
$$
is positive semidefinite by design. Hence the semimetric on $X_4$ is negative type.

Now, for a particular value of the parameter $s$, we can see that the small-scale limit of $d\Mag(tX_4)/dt$ becomes negative, so that $\Mag(tX_4)$ is decreasing near $t = 0$ and not monotonically increasing. To compute the small-scale limit of the derivative, we follow the idea of Roff and Yoshinaga in \cite{RY}: First of all, we can express the magnitude function in $t > 0$ as
$$
\Mag(tX_4) 
=
\langle \1, Z_{tX_4}^{-1} \1 \rangle
=
\frac{\langle \1, \mathrm{adj}(Z_{tX_4}) \1 \rangle}
{\det Z_{tX_4}},
$$
where $\mathrm{adj}(Z_{tX_4})$ means the adjugate. Then we consider the series expansions of $\langle \1, \mathrm{adj}(Z_{tX_4}) \1 \rangle$ and $\det Z_{tX_4}$ in $t$. In the present case, these expansions start at degree at least $4$ and have the following forms
\begin{align*}
\langle \1, \mathrm{adj}(Z_{tX_4}) \1 \rangle
&=
C t^4 + H t^5 + O(6), &
\det Z_{tX_4}
&=
C' t^4 + H't^5 + O(6),
\end{align*}
where $C, H, C', H'$ are polynomials in $d_{i, j}$, and $O(6)$ denotes the terms of degree $6$ and higher. A direct computation (by using computer software) gives us
\begin{align*}
C &= 64 s^2(3 + 22 s^2 + 27 s^4), &
H &=  -64 s^2 (-55 + 91 s^2 + 543 s^4 + 429 s^6),
\\
C' &= 128 s^2(1 + 8s^2 + 9s^4), &
H' &= -128 s^2 (-20 + 49 s^2 + 217 s^4 + 156 s^6).
\end{align*}
Thus, if $s \neq 0$, then l'H\^{o}pital's rule yields
$$
\lim_{t \to 0} \Mag(tX_4) 
=
\frac{C}{C'} 
=
\frac{3 + 22 s^2 + 27 s^4}{2(1 + 8s^2 + 9s^4)}.
$$
Under the assumption $s \neq 0$, the derivative $d\Mag(tX_4)/dt$ becomes
$$
\frac{d\Mag(tX_4)}{dt}
=
\frac{d}{dt}
\left(
\frac{C + Ht + O(2)}{C' + H't + O(2)}
\right)
=
\frac{HC' - CH' + O(1)}{(C' + H't + O(2))^2}.
$$
Therefore the small-scale limit of the derivative turns out to be
$$
\lim_{t \to 0}\frac{d\Mag(tX_4)}{dt}
=
\frac{HC' - CH'}{(C')^2}
=
\frac{(1 + s^2)(-5 + 61 s^2 + 352s^4 + 621 s^6 + 351 s^8)}
{2(1 + 8s^2 + 9s^4)^2}.
$$
This tends to $-5/2$ as $s \to 0$, and hence is negative when $s \neq 0$ is small enough. As a result, $\Mag(tX_4)$ is not monotonically increasing in $t$ for such an $s$.

\subsection{Examples of $5$-point semimetric spaces}
\label{subsec:5_point_example}

As in the $4$-point case, we construct a semimetric $d$ on a $5$-point set $X_5 = \{ 0, 1, 2, 3, 4 \}$ by taking $5$ points in $\R^2$ as
\begin{align*}
v_0
&=
\left(
\begin{array}{c}
0 \\ 0
\end{array}
\right),
&
v_1 
&=
\left(
\begin{array}{c}
1 \\ 0
\end{array}
\right),
&
v_2 
&=
\left(
\begin{array}{c}
0 \\ 1
\end{array}
\right),
&
v_3 
&=
\left(
\begin{array}{c}
1 \\ 1
\end{array}
\right),
&
v_4 
&=
\left(
\begin{array}{c}
s \\ s
\end{array}
\right),
\end{align*}
where $s \in \R$ is a parameter. Then $d_{i, j} = \lVert v_i - v_j \rVert^2$ for $i, j \in X_5$ gives rise to a semimetric on $X_5$ if and only if $s \neq 0, 1$. The distance matrix is 
$$
D 
=
\left(
\begin{array}{ccccc}
0 & 1 & 1 & 2 & 2s^2 \\
1 & 0 & 2 & 1 & s^2 + (1 - s)^2 \\
1 & 2 & 0 & 1 & s^2 + (1 - s)^2 \\
2 & 1 & 1 & 0 & 2(1 - s)^2 \\
2s^2 & s^2 + (1 - s)^2 & s^2 + (1 - s)^2 & 2(1 - s)^2 & 0
\end{array}
\right).
$$
As is mentioned in Remark \ref{rem:non_obtuse}, any non-obtuse set in $\R^2$ contains at most $2^2 = 4$ points. Therefore $\{ v_0, v_1, \ldots, v_4 \} \subset \R^2$ cannot be a non-obtuse set, and $d$ on $X_5$ is not a metric by Lemma \ref{lem:metric_condition}. One can also verify this directly by computing
\begin{align*}
d_{0, 1} + d_{1, 4} - d_{0, 4}
&=
1 + s^2 + (1 - s)^2 - 2s^2 = 2(1 - s), 
\\
d_{0, 4} + d_{4, 3} - d_{0, 3}
&=
2s^2 + 2(1 - s)^2 - 2
= 4s(s - 1), 
\\
d_{4, 0} + d_{0, 3} - d_{4, 3}
&=
2s^2 + 2 - 2(1-s)^2
= 4s.
\end{align*}
By construction, the $0$-th excess matrix $E^{(0)} = 2 \mathrm{Gram}(v_1, \ldots, v_4)$ is positive semidefinite. Hence the semimetric is negative type.

In the present case, we have
\begin{align*}
\langle \1, \mathrm{adj}(Z_{tX_5}) \1 \rangle
&=
H t^6 + J t^7 + O(8), &
\det Z_{tX_5}
&=
H' t^6 + J' t^7 + O(8),
\end{align*}
where $H, J, H', J'$ are given by
\begin{align*}
H &= 128 s^2(s - 1)^2, &
J &=  -\frac{320}{3} s^2(s - 1)^2(2s^2 - 2s + 5), \\
H' &= 64 s^2(s - 1)^2, &
J' &= -\frac{64}{3} s^2(s - 1)^2(4s^2 - 4s + 13).
\end{align*}
It follows that the derivative $d\Mag(tX_5)/dt$ is
$$
\frac{d\Mag(tX_5)}{dt}
=
\frac{d}{dt}
\left(
\frac{H + Jt + O(2)}{H' + J't + O(2)}
\right)
=
\frac{JH' - HJ' + O(1)}{(H' + J't + O(2))^2},
$$
and its small-scale limit is
$$
\lim_{t \to 0}\frac{d\Mag(tX_5)}{dt}
=
\frac{JH' - HJ'}{(H')^2}
=
\frac{1 + 2s - 2s^2}{3}.
$$
This is negative when $s < (1 - \sqrt{3})/2$ or $(1 + \sqrt{3})/2 < s$. For example, if we take $s = -1$ or $s = 2$, then $\lim_{t \to 0} d\Mag(tX_5)/dt = -1 < 0$. Consequently, $\Mag(tX_5)$ is not monotonically increasing in $t$ for such an $s$.

\subsection{The case with more points}

\begin{thm} \label{thm:more_points}
For any $n \ge 6$, there exists an $n$-point semimetric space $X_n$ of negative type with the following properties:
\begin{itemize}
\item
$X_n$ is not a metric space;

\item
We have $\lim_{t \to 0}d\Mag(tX_n)/dt < 0$, and hence $\Mag(tX_n)$ is not monotonically increasing in $t > 0$.

\end{itemize}
\end{thm}

\begin{proof}
We make use of a construction of finite (semi)metric spaces, which is called the $1$-sum operation in \cite{DL}: Let $(A, d_A)$ and $(B, d_B)$ be finite semimetric spaces. We choose and fix base points $a_0 \in A$ and $b_0 \in B$. Let $A \vee B$ be the set given by identifying the base points, which is called the wedge sum in topology. By design, if the cardinalities of $A$ and $B$ are $\# A$ and $\# B$, respectively, then $\# (A \vee B) = \# A + \# B - 1$. We then define a semimetric $d$ on $A \vee B$ as follows.
\begin{itemize}
\item[(i)]
If $a, a' \in A \subset A \vee B$, then $d(a, a') = d_A(a, a')$. Similarly, if $b, b' \in B \subset A \vee B$, then $d(b, b') = d_B(b, b')$.

\item[(ii)]
If $a \in A \subset A \vee B$ and $b \in B \subset A \vee B$, then $d(a, b) = d_A(a, a_0) + d_B(b_0, b)$. Similarly, $d(b, a)$ is defined.

\end{itemize}
A property of this construction is that, if $A$ and $B$ are negative type, then so is $A \vee B$. One can directly verify this by showing that an excess matrix of $A \vee B$ is the direct sum of those of $A$ and $B$. Another property of this construction is that the condition for the inclusion-exclusion axiom in \cite{L} is satisfied by $A, B \subset A \vee B$. As a result, we have $\Mag(t(A \vee B)) = \Mag(tA) + \Mag(tB) - 1$ for any $t > 0$.

Now, let $X_4$ be a $4$-point semimetric space of negative type such that:
\begin{itemize}
\item
$X_4$ is not a metric space;

\item
We have $\lim_{t \to 0}d\Mag(tX_4)/dt < 0$.

\end{itemize}
Such an example is constructed in \S\S\ref{subsec:4_point_example}. We also let a $2$-point set $X_2 = \{ 0, 1 \}$ be given a metric such that $d_{0, 1} = \ell > 0$. This is of negative type. We have
\begin{align*}
\Mag(tX_2) &= \frac{2}{1 + e^{- t \ell}}, &
\frac{d\Mag(tX_2)}{dt} &= \frac{2\ell e^{- t \ell}}{(1 + e^{- t \ell})^2}.
\end{align*}
For any positive integer $k$, we define $X_{4 + k}$ as the $1$-sum of $X_4$ and $k$ copies of $X_2$
$$
X_{4 + k} = X_4 \vee \underbrace{X_2 \vee \cdots \vee X_2}_{k}.
$$
This is a negative type semimetric space consisting of $4 + k$ points. From the formula $\Mag(tX_{4+k}) = \Mag(tX_4) + k \Mag(tX_2) - k$, it follows that 
$$
\lim_{t \to 0}\frac{d\Mag(tX_{4 + k})}{dt} 
=
\lim_{t \to 0}\frac{d\Mag(tX_4)}{dt} 
+ 
\frac{k\ell}{2}.
$$
Thus, if $\ell$ is small enough, then $\lim_{t \to 0} d\Mag(tX_{4 + k})/dt < 0$. Each $X_{4 + k}$ contains $X_4$ as a subspace. This implies that $X_{4 + k}$ is not a metric space.
\end{proof}

\begin{rem}
A finite (semi)metric space $X$ is said to be of \textit{strictly negative type} when its distance matrix $D$ is conditionally negative definite, or equivalently, its excess matrix $E_X$ is positive definite \cite{HLMT}. On the one hand, we can prove that $\lim_{t \to 0}d\Mag(tX)/dt = \det D/\langle \1, \mathrm{adj}(D) \1 \rangle> 0$ for any finite semimetric space $X$ of strictly negative type (cf.\ \cite{RW}). On the other hand, the examples of negative type semimetric spaces $X_n$ constructed in \S\S\ref{subsec:4_point_example}, \S\S\ref{subsec:5_point_example} and Theorem \ref{thm:more_points} have the property $\lim_{t \to 0}d\Mag(tX_n)/dt < 0$ and are not strictly negative type. Then one may ask whether there exists a semimetric space $X'_n$ with $n \ge 4$ such that:
\begin{itemize}
\item
$X'_n$ is not a metric space;

\item
There exists $t_0 > 0$ such that $d\Mag(tX'_n)/dt < 0$ at $t = t_0$, and hence $\Mag(tX'_n)$ is not monotonically increasing in $t > 0$;

\item
$X'_n$ is strictly negative type.

\end{itemize}
We can easily get such examples $X'_n$ by a perturbation of our examples $X_n$.
\end{rem}

\section{Further results for metric spaces}
\label{sec:metric}

This section is devoted to the proof of Theorem \ref{thm:sub_metric} in \S\ref{sec:introduction}.

\subsection{The monotonicity for $4$-point metric spaces}

We prove here Theorem \ref{thm:sub_metric} (a) in \S\ref{sec:introduction}:

\begin{thm} \label{thm:monotonicity_4_point_metric_space}
The magnitude function $t \mapsto \Mag(tX)$ of any $4$-point metric space $X$ is strictly monotonically increasing for all $t > 0$.
\end{thm}

To prove this theorem, we show the following two claims:
\begin{enumerate}
\item
A weighting of any $4$-point metric space has at most one negative entry. 

\item
Let $X$ be a finite metric space. If there exists $t > 0$ at which the zeta matrix $Z_{tX}$ is invertible and the weighting $w = Z_{tX}^{-1}\1$ has at most one negative entry, then the derivative of the magnitude function is positive at $t$.

\end{enumerate}

The first claim will be proved as Lemma \ref{lem:at_most_three_positive_entries} in a slightly more general form. After showing a preliminary lemma, we prove the second claim as Lemma \ref{lem:monotonicity_at_most_one_neg_weight_entry}. Using these claims, we prove Theorem \ref{thm:monotonicity_4_point_metric_space}.

\begin{lem} \label{lem:at_most_three_positive_entries}
Let $(X, d)$ be a metric space consisting of $n \ge 3$ points. If its zeta matrix $Z = (e^{ -d_{x, y}})_{x, y \in X}$ is positive definite, then the unique weighting $w = Z^{-1} \1$ has at least three positive entries.
\end{lem}

\begin{proof}
Assuming that $w = (w_x)_{x \in X}$ has at most two positive entries, we will derive a contradiction. Let $b, c \in X$ be two points such that $w_a \le 0$ for all $a \in X \backslash \{ b, c \}$. We put $B = \{ b, c \}$ and $A = X \backslash B$. We can then express the zeta matrix $Z$ and the weighting $w$ as 
\begin{align*}
Z &=
\left(
\begin{array}{cc}
Z_{AA} & Z_{AB} \\
Z_{BA} & Z_{BB}
\end{array}
\right),
&
w &=
\left(
\begin{array}{c}
w_A \\ w_B
\end{array}
\right).
\end{align*}
The defining equation of the weighting $Z w = \1$ is equivalent to
\begin{align*}
Z_{AA}w_A + Z_{AB}w_B &= \1_A, &
Z_{BA}w_A + Z_{BB}w_B &= \1_B,
\end{align*}
where $\1_A$ and $\1_B$ are vectors whose entries are $1$. Since $Z$ is positive definite, so are $Z_{AA}$ and $Z_{BB}$. Then, we can eliminate $w_B$ from the two equations above to get a single equation $S w_A = r$, where $S$ and $r$ are defined by
\begin{align*}
S &= Z_{AA} - Z_{AB} Z_{BB}^{-1} Z_{BA}, &
r &= \1_A - Z_{AB} Z_{BB}^{-1}\1_B.
\end{align*}
Being a Schur complement of $Z$, the matrix $S$ is positive definite. A direct computation gives
$$
Z_{BB}^{-1} \1_B =
\frac{1}{1 + z_{bc}}
\left(
\begin{array}{cc}
1 \\ 1
\end{array}
\right),
$$
where $z_{b,c} = e^{-d_{b, c}}$ is the $(b, c)$-entry of the zeta matrix $Z$. Thus, for each $a \in A$, we can directly compute the $a$-th entry of $r$ as follows 
$$
r_a = 1 - \frac{z_{a,b} + z_{a,c}}{1 + z_{b,c}}
=
\frac{1 + z_{b,c} - z_{a,b} - z_{a,c}}{1 + z_{b,c}}.
$$
The triangle inequality $d_{b, c} \le d_{b, a} + d_{a, c}$ implies $z_{b, c} \ge z_{a,b}z_{a,c}$. It follows that
$$
r_a \ge \frac{(1 - z_{a,b})(1 - z_{a,c})}{1 + z_{b,c}} > 0,
$$
since $a, b, c$ are distinct. Now, if $w_A \neq 0$, then $\langle w_A, S w_A \rangle > 0$ because $S$ is positive definite. However, we have $\langle w_A, S w_A \rangle = \langle w_A, r \rangle < 0$, since the entries of $w_A$ are non-positive and those of $r$ are positive. Hence we have $w_A = 0$. We then get $r = S w_A = 0$, which is a contradiction. 
\end{proof}

Let $(X, d)$ be a semimetric space consisting of $n \ge 2$ points. Assuming that $X = \{ 0, 1, \ldots, n - 1 \}$, we put 
\begin{align*}
a_i &= d_{0, i}, &
h_{i, j} &= a_i + a_j - d_{i, j}, &
z_i(t) &= e^{-ta_i}
\end{align*}
for $i, j \in \{ 1, \ldots, n - 1 \}$ and $t > 0$. It is clear that $a_i > 0$ and $0 < z_i(t) < 1$. By using the following expression of the zeta matrix $Z(t) = Z_{tX} = ( e^{-td_{x, y}} )$ and introducing a diagonal matrix $R(t)$
\begin{align*}
Z(t) 
&=
\left(
\begin{array}{cc}
1 & {}^t z \\
z & A
\end{array}
\right),
&
R(t)
&=
\mathrm{diag}(1 - z_1, \ldots, 1 - z_{n-1}),
\end{align*}
we define $K(t)$ as
$$
K(t) = R(t)^{-1}(A - z{}^tz) R(t)^{-1}.
$$

\begin{lem} \label{lem:some_formulae_for_monotonicity}
Let $d$ be a semimetric on a set $X = \{ 0, 1, \ldots, n-1 \}$ consisting of $n \ge 2$ points. Let $a_i, h_{i, j}, z_i, z, A, R(t), K(t)$ be as above.
\begin{enumerate}
\item[(a)]
If there exists $t > 0$ at which the zeta matrix $Z(t) = Z_{tX} = ( e^{-td_{x, y}} )$ is invertible, then $K(t)$ is also invertible at this $t$. Furthermore, if we define $u(t)$ by $u(t) = K(t)^{-1}\1$ and $w_B = {}^t(w_1, \ldots, w_{n-1})$ by omitting the first entry of $w = Z(t)^{-1}\1$, then 
\begin{align*}
u(t) &= R(t)w_B(t),
\end{align*}
and we can express the magnitude $\Mag(tX) = \langle \1, Z(t)^{-1}\1 \rangle$ and its derivative at this $t$ as
\begin{align*}
\Mag(tX) &= 1 + \langle \1, K(t)^{-1} \1 \rangle, &
\frac{d\Mag(tX)}{dt} &= - \langle u(t), \frac{dK(t)}{dt} u(t) \rangle.
\end{align*}

\item[(b)]
Suppose that $d$ is a metric. For any $i, j \in \{ 1, \ldots, n - 1 \}$, the $(i, j)$-entry of $K(t)$ 
$$
K_{i,j} =
\frac{e^{th_{i, j}} - 1}{(e^{ta_i} - 1)(e^{ta_j} - 1)}
$$
is non-negative, and the $i$-th diagonal entry is positive: $K_{i,i} > 0$. Furthermore, the derivative of the $(i, j)$-entry is non-positive
$$
\frac{d K_{i,j}}{dt} \le 0,
$$
and that of the $i$-th diagonal entry is negative: $dK_{i,i}/dt < 0$.

\end{enumerate}
\end{lem}

\begin{proof}
(a)
The matrix $A - z{}^tz$ is a Schur complement of $Z$, so it is invertible if $Z$ is invertible, as can be seen directly by
$$
Z =
\left(
\begin{array}{cc}
1 & {}^tz \\
z & A
\end{array}
\right)
=
\left(
\begin{array}{cc}
1 & 0 \\
z & I
\end{array}
\right)
\left(
\begin{array}{cc}
1 & 0 \\
0 & A - z{}^tz
\end{array}
\right)
\left(
\begin{array}{cc}
1 & {}^tz \\
0 & I
\end{array}
\right).
$$
The defining equation $Zw = \1$ is equivalent to
\begin{align*}
w_0 + \langle z, w_B \rangle &= 1, &
w_0z + Aw_B &= \1,
\end{align*}
where $w_0$ is the first entry of $w = (w_i)$. Eliminating $w_0 = 1 - \langle z, w_B \rangle$ from the second equation, one has 
$$
A w_B - \langle z, w_B \rangle z = \1 - z.
$$
Noting that $\1 - z = R \1$, we have
\begin{align*}
KRw_B 
&=
(R^{-1}(A - z{}^tz) R^{-1})(Rw_B)
=
R^{-1}(A - z{}^tz) w_B
\\
&=
R^{-1}(A w_B - \langle z, w_B \rangle z)
=
R^{-1}(\1 - z) = \1,
\end{align*}
which is equivalent to $Rw_B = K^{-1}\1 = u$. Now, the magnitude is computed as
\begin{align*}
\Mag(tX) 
&= w_0 + \langle \1, w_B \rangle
=
1 - \langle z, w_B \rangle + \langle \1, w_B \rangle
=
1 + \langle \1 - z, w_B \rangle
\\
&=
1 + \langle R \1, w_B \rangle
=
1 + \langle \1, R w_B \rangle
=
1 + \langle \1, K^{-1}\1 \rangle.
\end{align*}
From this and the formula of the derivative $dK^{-1}/dt = -K^{-1}(dK/dt)K^{-1}$, we directly get the remaining expression of $d\Mag(tX)/dt$.

(b)
For any $i, j \in \{ 1, \ldots, n - 1 \}$, the $(i, j)$-entry of $A$ is $e^{-td_{i, j}}$. Thus, we use the definition $h_{i, j} = a_i + a_j - d_{i, j}$ to get
\begin{align*}
K_{i,j}
&=
\frac{e^{-td_{i, j}} - e^{-t(a_i + a_j)}}{(1 - e^{-ta_i})(1 - e^{-ta_j})}
=
\frac{e^{th_{i, j}} - 1}{(e^{ta_i} - 1)(e^{ta_j} - 1)}.
\end{align*}
Since $d$ is a metric, the triangle inequality implies $h_{i, j} \ge 0$. Hence $K_{i,j} \ge 0$. We have $h_{i, i} = 2a_i > 0$, which implies $K_{i,i} > 0$. To get the remaining estimate of $K'_{i,j} = dK_{i,j}/dt$, we consider a function 
$$
f_t(x) = \frac{x}{1 - e^{-tx}}
$$
of $x > 0$. This is strictly monotonically increasing, since its derivative
$$
\frac{df_t}{dx}
=
\frac{1 - (1 + tx)e^{-tx}}{(1 - e^{-tx})^2}
$$
is positive by $e^{tx} > (1 + tx)$ for $tx > 0$. For $a, b > 0$, we can see
\begin{align*}
f_t(a + b)
&=
\frac{a}{1 - e^{-t(a + b)}} + \frac{b}{1 - e^{-t(a + b)}}
<
\frac{a}{1 - e^{-ta}} + \frac{b}{1 - e^{-tb}}
=
f_t(a) + f_t(b).
\end{align*}
Now, in the case where $h_{i, j} = 0$, we have $K_{i,j} = 0$, and hence $K'_{i,j} = 0$. In the case where $h_{i, j} \neq 0$, we have $K_{i,j} > 0$, and
\begin{align*}
\frac{K'_{i,j}}{K_{i,j}}
&=
f_t(h_{i, j}) - f_t(a_i) - f_t(a_j)
\le
f_t(a_i + a_j) - f_t(a_i) - f_t(a_j)
\le 0.
\end{align*}
Finally, in the case that $i = j$, we have
\begin{align*}
K_{i,i} &= \frac{e^{ta_i} + 1}{e^{ta_i} - 1}, &
K'_{i,i} &= \frac{-2 a_i e^{ta_i}}{(e^{ta_i} - 1)^2} < 0,
\end{align*}
which completes the proof.
\end{proof}

\begin{lem} \label{lem:monotonicity_at_most_one_neg_weight_entry}
Let $(X, d)$ be a metric space consisting of $n \ge 2$ points. Suppose there exists $t > 0$ at which the zeta matrix $Z_{tX}$ is invertible and the weighting $Z_{tX}^{-1}\1$ has at most one negative entry. Then the derivative of the magnitude function $\Mag(tX) = \langle \1, Z_{tX}^{-1} \1 \rangle$ at this $t$ is positive
$$
\frac{d\Mag(tX)}{dt} > 0.
$$
\end{lem}

\begin{proof}
Without loss of generality, we can assume that $X = \{ 0, 1, \ldots, n - 1 \}$ and $w_i \ge 0$ for $i = 1, \ldots, n - 1$, where $w_i$ denotes the $i$-th entry of the unique weighting $w = Z_{tX}^{-1}\1$. According to this assumption, the entries $u_i$ of the vector $u = R w_B$ in Lemma \ref{lem:some_formulae_for_monotonicity} (a) are non-negative. By Lemma \ref{lem:some_formulae_for_monotonicity} (b), the entries $K'_{i,j}$ of $K' = dK/dt$ are non-positive and the diagonal entries $K'_{i,i}$ are negative. Thus, the expression of the derivative of the magnitude function in Lemma \ref{lem:some_formulae_for_monotonicity} (a) gives us
$$
\frac{d\Mag(tX)}{dt} =
- \langle u, K'(t) u \rangle
=
\sum_{i, j = 1}^{n-1} u_i u_j (-K'_{i,j})
\ge
\sum_{i = 1}^{n-1} u_i^2 (- K'_{i,i}).
$$
The last term is positive, because the formula $Ku = \1$ from the definition of $u$ implies the existence of $i \in \{ 1, \ldots, n - 1 \}$ such that $u_i \neq 0$.
\end{proof}

\begin{proof}[The proof of Theorem \ref{thm:monotonicity_4_point_metric_space}]
As is well known, any $4$-point metric space $X$ is of negative type, and hence $Z_{tX}$ is positive definite for all $t > 0$. Thus, by Lemma \ref{lem:at_most_three_positive_entries}, the unique weighting of $tX$ has at most one negative entry. Therefore we can apply Lemma \ref{lem:monotonicity_at_most_one_neg_weight_entry} to show $d\Mag(tX)/dt > 0$ for all $t > 0$.
\end{proof}

\begin{rem}
In general, there exists a finite metric space such that two negative entries appear in its weighting and the magnitude function is monotonically increasing. For instance, consider a distance $d$ on a $5$-point set whose distance matrix and excess matrix are
\begin{align*}
D &=
\left(
\begin{array}{ccccc}
0 & 1 & 1 & 1 & 1 \\
1 & 0 & 1 & 1 & 1 \\
1 & 1 & 0 & 2 & 2 \\
1 & 1 & 2 & 0 & 2 \\
1 & 1 & 2 & 2 & 0
\end{array}
\right),
&
E &=
\left(
\begin{array}{cccc}
2 & 1 & 1 & 1 \\
1 & 2 & 0 & 0 \\
1 & 0 & 2 & 0 \\
1 & 0 & 0 & 2
\end{array}
\right).
\end{align*}
Since the non-zero entries of $D$ are $1$ or $2$, the triangle inequalities are satisfied by $d$. In view of Sylvester's criterion and $\det E = 4 > 0$, we see that $E$ is positive definite and hence $d$ is of (strictly) negative type. For $t > 0$, the unique weighting $w$ is given by
$$
{}^tw = 
\frac{1}{1 + 2q - 2q^2}
\big(
1 - 2q, 1 - 2q, 1, 1, 1
\big),
$$
where $q = e^{-t}$. We have $1 + 2q - 2q^2 = 1 + 2q(1 - q) > 0$. Thus, $w$ has two negative entries if $t < \log 2$. The magnitude function and its derivative are
\begin{align*}
\Mag(tX) &= \frac{5 - 4q}{1 + 2q - 2q^2}, &
\frac{d\Mag(tX)}{dt} &= \frac{2q(7 - 10q + 4q^2)}{(1 + 2q - 2q^2)^2}.
\end{align*}
Because $7 - 10q + 4q^2 = 4 \big( q - \frac{5}{4} \big)^2 + \frac{3}{4} > 0$, this magnitude function is (strictly) monotonically increasing for all $t > 0$.
\end{rem}

\subsection{$12$-point metric spaces without monotonicity}

We prove here Theorem \ref{thm:sub_metric} (b) in \S\ref{sec:introduction}, by constructing $12$-point metric spaces of negative type whose magnitude functions are not monotonically increasing. Our examples generalize the idea of the construction in \cite{K}.

\smallskip

Let $A = \{ a_0, \ldots, a_5 \}$ and $B = \{ b_0, \ldots, b_5 \}$ be $6$-point sets. As their union, we define a $12$-point set $X = A \cup B = \{ a_0, \ldots, a_5, b_0, \ldots, b_5 \}$. In what follows, the indices of the points $a_i$ and $b_i$ are taken modulo $6$. For $k \in \R$, we define vectors $R^A = (R^A_i), R^B = (R^B_i)$ and $R^C = (R^C_i)$ by
\begin{align*}
R^A &= (0, 5k-1, 6k, 5k-1, 6k, 5k-1), \\
R^B &= (0, 5k+3, 6k, 5k+3, 6k, 5k+3), \\
R^C &= (4k+1, 3k, 4k+1, 3k, 4k+1, 9k).
\end{align*}
For $i, j \in \Z/6$, we define a distance function $d : X \times X \to \R$ by
\begin{align*}
d(a_i, a_j) &= R^A_{j - i}, &
d(b_i, b_j) &= R^B_{j - i}, &
d(a_i, b_j) &= d(b_j, a_i) = R^C_{j-i}.
\end{align*}
Thus, the distance matrix $D = (d(x, y))$ can be described as
{\scriptsize
$$
D=
\left(
\begin{array}{c@{}c@{}c@{}c@{}c@{}c|c@{}c@{}c@{}c@{}c@{}c}
0 & 5k-1 & 6k & 5k-1 & 6k & 5k-1 & 4k+1 & 3k & 4k+1 & 3k & 4k+1 & 9k \\
5k-1 & 0 & 5k-1 & 6k & 5k-1 & 6k & 9k & 4k+1 & 3k & 4k+1 & 3k & 4k+1 \\
6k & 5k-1 & 0 & 5k-1 & 6k & 5k-1 & 4k+1 & 9k & 4k+1 & 3k & 4k+1 & 3k \\
5k-1 & 6k & 5k-1 & 0 & 5k-1 & 6k & 3k & 4k+1 & 9k & 4k+1 & 3k & 4k+1 \\
6k & 5k-1 & 6k & 5k-1 & 0 & 5k-1 & 4k+1 & 3k & 4k+1 & 9k & 4k+1 & 3k \\
5k-1 & 6k & 5k-1 & 6k & 5k-1 & 0 & 3k & 4k+1 & 3k & 4k+1 & 9k & 4k+1 \\
\hline
4k+1 & 9k & 4k+1 & 3k & 4k+1 & 3k & 0 & 5k+3 & 6k & 5k+3 & 6k & 5k+3 \\
3k & 4k+1 & 9k & 4k+1 & 3k & 4k+1 & 5k+3 & 0 & 5k+3 & 6k & 5k+3 & 6k \\
4k+1 & 3k & 4k+1 & 9k & 4k+1 & 3k & 6k & 5k+3 & 0 & 5k+3 & 6k & 5k+3 \\
3k & 4k+1 & 3k & 4k+1 & 9k & 4k+1 & 5k+3 & 6k & 5k+3 & 0 & 5k+3 & 6k \\
4k+1 & 3k & 4k+1 & 3k & 4k+1 & 9k & 6k & 5k+3 & 6k & 5k+3 & 0 & 5k+3 \\
9k & 4k+1 & 3k & 4k+1 & 3k & 4k+1 & 5k+3 & 6k & 5k+3 & 6k & 5k+3 & 0
\end{array}
\right).
$$
}

\begin{lem} \label{lem:12_point_metric}
The distance function $d$ on $X$ is a metric if and only if $k \ge 1$.
\end{lem}

\begin{proof}
If $d$ is a metric, then we have
$$
5k + 3 \le d(b_0, b_1) \le d(b_0, a_0) + d(a_0, b_1) = 7k + 1.
$$
Hence we have $k \ge 1$. Conversely, let us assume $k \ge 1$. We then verify that all the triangle inequalities are satisfied. To this end, we list sets of distances $\{ d(x, y), d(y, z), d(x, z) \}$ associated to three distinct points $\{ x, y, z \}$ in $X$. 
\begin{enumerate}
\item
In the case where $\{ x, y, z \} = \{ a_i, a_j, a_\ell \} \subset A$, there are four cases of the indices $i, j, \ell$.
\begin{enumerate}
\item
The case that $i, j, \ell$ are even. (For instance, $i = 0, j = 2, \ell = 4$.) Then we only have a set of distances $\{ d(a_i, a_j), d(a_j, a_\ell), d(a_i, a_\ell) \} = \{ 6k, 6k, 6k \}$.

\item
In the case that $i, j$ are even and $\ell$ is odd, we have a set of distances $\{ 5k-1, 5k-1, 6k \}$. 

\item
In the case that $i$ is even and $j, \ell$ are odd, we have a set of distances $\{ 5k-1, 5k-1, 6k \}$.

\item
In the case that $i, j, \ell$ are odd, we have a set of distances $\{ 6k, 6k, 6k \}$. 

\end{enumerate}

\item
In the case where $\{ x, y, z \} = \{ b_i, b_j, b_\ell \} \subset B$, the same consideration as above shows that the possible sets of distances are $\{ 6k, 6k, 6k \}$ and $\{ 5k+3, 5k+3, 6k \}$.

\item
The case where $\{x, y \} = \{ a_i, a_j \} \subset A$ and $z = b_\ell \in B$. Then we have four cases of the indices $i, j, \ell$.
\begin{enumerate}
\item
In the case that $i, j, \ell$ are even, we have a set of distances $\{ 4k+1, 4k+1, 6k \}$. 

\item
In the case that $i, j$ are even and $\ell$ is odd, we have sets of distances $\{ 3k, 3k, 6k \}$ and $\{ 3k, 6k, 9k \}$. 

\item
In the case that $i$ is even and $j, \ell$ are odd, we have sets of distances $\{ 3k, 4k+1, 5k-1 \}$ and $\{ 4k+1, 5k-1, 9k \}$.

\item
In the case that $i, j, \ell$ are odd, we have a set of distances $\{ 4k+1, 4k+1, 6k \}$.

\end{enumerate}

\item
The case where $x = a_i \in A$ and $\{ y, z \} = \{ b_j, b_\ell \} \subset B$. Then we have four cases of the indices $i, j, \ell$.
\begin{enumerate}
\item
In the case that $i, j, \ell$ are even, we have a set of distances $\{ 4k+1, 4k+1, 6k \}$.

\item
In the case that $i, j$ are even and $\ell$ is odd, we have sets of distances $\{ 3k, 4k+1, 5k+3 \}$ and $\{ 4k+1, 5k+3, 9k \}$.

\item
In the case that $i$ is even and $j, \ell$ are odd, we have sets of distances $\{ 3k, 3k, 6k \}$ and $\{ 3k, 6k, 9k \}$.

\item
In the case that $i, j, \ell$ are odd, we have a set of distances $\{ 4k+1, 4k+1, 6k \}$.

\end{enumerate}
\end{enumerate}
Now, we can verify that the triangle inequalities hold whenever $k \ge 1$.
\end{proof}

\begin{lem}
The metric space $(X, d)$ is negative type if and only if $k \ge 1$.
\end{lem}

\begin{proof}
Let $D \in M(12, \R)$ be the distance matrix of $d$ as displayed above Lemma \ref{lem:12_point_metric}, which can be seen as a block matrix
$$
D =
\left(
\begin{array}{cc}
D_A & D_C \\
{}^tD_C & D_B
\end{array}
\right).
$$
Corresponding to this block decomposition, we write a vector $z \in \R^{12}$ such that $\langle z, \1 \rangle = 0$ as $z = \left(\begin{array}{c} u \\ v \end{array}\right)$ in terms of $u = (u_i)_{i = 0}^5, v = (v_i)_{i = 0}^5 \in \R^6$. Define $\epsilon \in \R^6$ by $\epsilon = {}^t(1, -1, 1, -1, 1, -1)$, and put
\begin{align*}
s &= \frac{1}{6} \sum_{i = 0}^5 u_i = -\frac{1}{6} \sum_{i = 0}^5 v_i, &
\alpha &= \frac{1}{6} \sum_{i = 0}^5 (-1)^i u_i, &
\beta &= \frac{1}{6} \sum_{i = 0}^5 (-1)^i v_i.
\end{align*}
We can then decompose $u, v$ as
\begin{align*}
u &= s \1 + \alpha \epsilon + u', &
v &=  - s \1 + \beta \epsilon + v',
\end{align*}
where $u' = (u'_i)_{i = 0}^5 \in \R^6$ and $v' = (v'_i)_{i = 0}^5 \in \R^6$ are orthogonal to both $\1$ and $\epsilon$. As a result, we have
\begin{align*}
\langle z, Dz \rangle
&=
\langle 
\left(\begin{array}{c}  
\! s \1 \! \\ \! -s \1 \!
\end{array}\right) , 
D \left(\begin{array}{c} \! s \1 \! \\ \! -s \1 \! \end{array}\right)
\rangle
+
\langle 
\left(\begin{array}{c}
 \! \alpha \epsilon \! \\ \! \beta \epsilon \! 
\end{array}\right) , 
D \left(\begin{array}{c}
\! \alpha \epsilon \! \\ \! \beta \epsilon \!
\end{array}\right)
\rangle
+
\langle 
\left(\begin{array}{c} \! u'\!  \\ \! v' \! \end{array}\right) , 
D \left(\begin{array}{c} \! u'\!  \\ \! v' \! \end{array}\right)
\rangle.
\end{align*}
The vectors $\1$ and $\epsilon$ are eigenvectors of $D_A, D_B, D_C$ and ${}^tD_C$: 
\begin{align*}
D_A\1 &= (27k-3)\1, &
D_B\1 &= (27k+9)\1, &
D_C\1 &= {}^tD_C \1 = (27k + 3)\1, 
\\
D_A\epsilon &= -3(k-1)\epsilon, &
D_B\epsilon &= -3(k+3)\epsilon, &
D_C\epsilon &= {}^tD_C\epsilon = -3(k-1)\epsilon.
\end{align*}
Using these formulae, we get
\begin{align*}
\langle 
\left(\begin{array}{c} s \1 \\ -s \1 \end{array}\right) , 
D \left(\begin{array}{c} s \1 \\ -s \1 \end{array}\right)
\rangle
&=0,
\\
\langle 
\left(\begin{array}{c} \alpha \epsilon \\ \beta \epsilon \end{array}\right) , 
D \left(\begin{array}{c} \alpha \epsilon \\ \beta \epsilon \end{array}\right)
\rangle
&=
-18(k-1)(\alpha + \beta)^2 - 72\beta^2.
\end{align*}
Note that $\langle u', \1 \rangle = \langle u', \epsilon \rangle = 0$ is equivalent to $u'_0 + u'_2 + u'_4 = u'_1 + u'_3 + u'_5 = 0$. The vector $v'$ also has the same property. Using these properties, we can see
\begin{align*}
D_A u' &= -6k u', &
D_B v' &= -6kv'.
\end{align*}
Similarly, we can see that the $i$-th entry of $D_Cv'$ is computed as
$$
(D_Cv')_i 
= 
(4k+1) \sum_{j = i \mod 2} v'_j 
+ 3k \sum_{j \neq i \mod 2} v'_j
+ 6k v'_{i-1}
= 
6k v'_{i-1}.
$$
It follows that
\begin{align*}
\langle 
\left(\begin{array}{c} u' \\ v' \end{array}\right) , 
D \left(\begin{array}{c} u' \\ v' \end{array}\right)
\rangle
&=
-6k \sum_{i = 0}^5 (u'_i)^2 - 6k \sum_{i = 0}^5 (v'_i)^2
+12k \sum_{i = 0}^5 u'_iv'_{i-1}
\\
&=
-6k \sum_{i = 0}^5 (u'_i - v'_{i-1})^2.
\end{align*}
Putting the computations above together, we get
$$
-\langle z, Dz \rangle
=
6k \sum_{i = 0}^5 (u'_i - v'_{i-1})^2
+
18(k- 1)(\alpha + \beta)^2 
+ 72\beta^2.
$$
This expression shows that the metric $d$ is negative type if and only if $k \ge 1$.
\end{proof}

\begin{lem}
For $k \ge 1$, the magnitude function of the $12$-point metric space $(X, d)$ is expressed as
$$
\Mag(tX) = \frac{f(t) + g(t) - 2h(t)}{f(t)g(t) - h(t)^2},
$$
where $f(t), g(t)$ and $h(t)$ are given by
\begin{align*}
f(t) &= \frac{1 + 2 e^{-6kt} + 3e^{-(5k-1)t}}{6}, \\
g(t) &= \frac{1 + 2 e^{-6kt} + 3e^{-(5k+3)t}}{6}, \\
h(t) &= \frac{3e^{-(4k+1)t} + 2 e^{-3kt} + e^{-9kt}}{6}.
\end{align*}
In particular,  we have the following limits
\begin{align*}
\lim_{t \to 0}\Mag(tX) 
&=
\frac{k+2}{k+1},
\\
\lim_{t \to 0}\frac{d\Mag(tX)}{dt}
&=
\frac{k(-29k^2 + 78k + 51)}{6(k+1)^2}.
\end{align*}
\end{lem}

\begin{proof}
By definition, the simultaneous cyclic shift of the indices of points in $A$ and $B$ leaves the distance matrix $D = ( d(x, y) )$, and hence the zeta matrix $Z = (e^{-t d(x, y)})$, invariant. By uniqueness, the weighting $w = (w_x)_{x \in X}$ of $Z$ is constant on each of $A$ and $B$. That is, $w_a = w_{a'}$ for $a, a' \in A$ and $w_b = w_{b'}$ for $b, b' \in B$. Put $u = \sum_{a \in A} w_a$ and $v = \sum_{b \in B} w_b$. Then the defining equation $Z w = \1$ for $w$ is equivalent to
$$
\left(
\begin{array}{cc}
f & h \\
h & g
\end{array}
\right)
\left(
\begin{array}{c}
u \\ v
\end{array}
\right)
=
\left(
\begin{array}{c}
1 \\ 1
\end{array}
\right).
$$
The solution is given by $u = (g - h)/(fg - h^2)$ and $v = (f - h)/(fg - h^2)$. Thus,
$$
\Mag(tX) = u + v = \frac{f + g - 2h}{fg - h^2}.
$$
The remaining limits can be computed by l'H\^{o}pital's rule.
\end{proof}

\begin{thm} \label{thm:12_point_example}
There exists a $12$-point metric space of negative type whose magnitude function is not monotonically increasing.
\end{thm}

\begin{proof}
Let $\phi(k) = -29k^2 + 78k + 51$ be the function in the numerator of the expression of $\lim_{t \to 0}d\Mag(tX)/dt$. It suffices to show that there exists $k \ge 1$ such that $\phi(k) < 0$. The quadratic equation $\phi(k) = 0$ has the solutions
$$
k_\pm = \frac{39 \pm 10\sqrt{30}}{29}.
$$
We have $k_- < 0$ and $1 < k_+ = 3.23353\cdots$. Thus, $\phi(k) < 0$ for any $k > k_+$. For instance, if we choose $k = 4$, then $\phi(4) = -101 < 0$.
\end{proof}


\appendix

\section{The monotonicity of the determinant function}
\label{sec:appendix:determinant}

We prove a result which is mentioned in \S\ref{sec:introduction}.

\begin{prop}
Let $X$ be a finite semimetric space of negative type. Then the determinant function $t \mapsto \det Z_{tX}$ is monotonically increasing for all $t > 0$.
\end{prop}

\begin{proof}
By assumption, $Z_{tX}$ is positive definite and $\det Z_{tX} > 0$. We have
\begin{align*}
\frac{1}{\det Z_{tX}}
\frac{d \det Z_{tX}}{dt}
&=
\tr \left(
\frac{dZ_{tX}}{dt}
Z_{tX}^{-1}
\right)
=
\tr
\left(
(-D \circ Z_{tX})Z_{tX}^{-1}
\right)
\\
&=
\tr
\left(
-D(Z_{tX} \circ Z_{tX}^{-1})
\right)
=
\mathrm{tr}
\left(
-D(Z_{tX} \circ Z_{tX}^{-1} - I)
\right),
\end{align*}
where $\circ$ denotes the Hadamard (entrywise) product, and $I \in M(n, \R)$ the identity matrix. Once the last term above is shown to be non-negative, the proof will be completed. Let $P = I - \frac{1}{n}\1 {}^t\1$ be the orthogonal projection onto the orthogonal complement $\Pi_0 = \1^\perp$ of the vector $\1 \in \R^n$ whose entries are $1$. A computation shows that $(Z_{tX} \circ Z_{tX}^{-1})\1 = \1$. Using this formula and $I \1 = \1$, we can see
$$
P (Z_{tX} \circ Z_{tX}^{-1} - I) P = Z_{tX} \circ Z_{tX}^{-1} - I.
$$
Hence we get
$$
\tr
\left(
-D(Z_{tX} \circ Z_{tX}^{-1} - I)
\right)
=
\tr
\left(
(-PDP)(Z_{tX} \circ Z_{tX}^{-1} - I)
\right).
$$
This is non-negative for the following reason: By definition $(X, d)$ is of negative type if and only if $\langle v, Dv \rangle \le 0$ for all $v \in \Pi_0$. Since $P$ is the orthogonal projection onto $\Pi_0$, it is clear that the negative type assumption is equivalent to the positive semidefiniteness of $-PDP$.  By Fiedler's theorem (see \cite{BK} for example), if $A$ is positive definite, then $A \circ A^{-1} - I$ is positive semidefinite. Applying this to $A = Z_{tX}$, we find that $Z_{tX} \circ Z_{tX}^{-1} - I$ is positive semidefinite. Finally, if $A$ and $B$ are positive semidefinite matrices of the same size, then $\tr(AB) \ge 0$ by elementary linear algebra. 
\end{proof}


\end{document}